\documentclass{amsart}
\usepackage{setspace}

\usepackage{amsthm}
\usepackage{a4}
\usepackage{amsfonts}
\usepackage{amsmath}
\usepackage{amssymb}
\usepackage{mathrsfs}
\usepackage{enumerate}

\usepackage{fancyhdr}
\usepackage{color}

\parskip\medskipamount
\DeclareFontFamily{OT1}{manual}{}
\DeclareFontShape{OT1}{manual}{m}{n}{ <10> manfnt }{}

\newcommand*{\ch}{\mathrm{char}}
\newcommand*{\sub}{\underline{s}}

\newcommand*{\sgn}{\mathrm{sgn}}

\newcommand*{\<}{\langle}
\renewcommand*{\>}{\rangle}
\newcommand*{\x}{\times}

\newcommand*{\snq}{\langle 1\rangle\perp n \times q}
\newcommand*{\snlq}{\langle 1\rangle\perp (n-1) \times q}

\newcommand{\leqs}{\leqslant}
\newcommand{\geqs}{\geqslant}

\newtheorem{lemma}{Lemma}[section]
\newtheorem{theorem}[lemma]{Theorem}
\newtheorem{prop}[lemma]{Proposition}
\newtheorem{cor}[lemma]{Corollary}

\theoremstyle{definition}

\theoremstyle{remark}
\newtheorem{remark}[lemma]{Remark}

\theoremstyle{definition}
\newtheorem{example}[lemma]{Example}

\theoremstyle{definition}
\newtheorem{question}[lemma]{Question}

\title{Isotropy of multiples of Pfister forms and weak isotropy of forms over field extensions}
\author{James O'Shea}
\address{James O'Shea,\newline School of Mathematical Sciences, University College Dublin, Dublin 4, Ireland.\newline E-mail: james.oshea@ucd.ie}
\date{}
\begin{document}

%
%

\maketitle

\begin{abstract}\noindent The isotropy of multiples of Pfister forms is studied. In particular, an improved lower bound on the value of their first Witt index is obtained. This result and certain of its corollaries are applied to the study of the weak isotropy index (or equivalently, the sublevel) of arbitrary quadratic forms. The relationship between this invariant and the level of the (quadratic) form is investigated. The problem of determining the set of values of the weak isotropy index of a form with respect to field extensions is addressed, with both admissible and inadmissible integers being determined. The analogous investigation with respect to the level of a form is also undertaken, with some questions asked by Berhuy, Grenier-Boley and Mahmoudi being resolved. An examination of the weak isotropy index and the level of round and Pfister forms concludes this article.
\end{abstract}

\section{Introduction}

The subject matter of this article is drawn from three distinct, but interrelated, topics within quadratic form theory: the isotropy of multiples of Pfister forms; the weak isotropy index of forms; the level of forms.

Given the central role of Pfister forms within the theory of quadratic forms, the isotropy behaviour of their products with other forms has been a topic of long-standing interest. In particular, given forms $\pi$ and $q$ over a field $F$ such that $\pi$ is similar to an anisotropic Pfister form, in the case where the product of $\pi$ and $q$ is isotropic over $F$, it has long been known that the Witt index of this product is a multiple of the dimension of $\pi$. Thus, if the product of $\pi$ and $q$ is anisotropic over $F$, it follows that the Witt index of this product over the generic extension that makes it isotropic (its first Witt index) is at least the dimension of $\pi$. 

Our main result in the opening section of this article, Theorem~\ref{i1bound}, improves upon this lower bound, establishing that the first Witt index of the product of $\pi$ and $q$ is at least the first Witt index of $q$ times the dimension of $\pi$. Moreover, in Proposition~\ref{i1maxspl}, we determine that this inequality is actually an equality in the case where the form $q$ has maximal splitting, and thus, as a result, we can conclude that maximal splitting is preserved with respect to taking products with Pfister forms.



A form $q$ over $F$ is weakly isotropic if there is $n\in\mathbb{N}\setminus \{ 0\}$ such that the orthogonal sum of $n$ copies of $q$ is isotropic over $F$. The \emph{weak isotropy index} of a form $q$, denoted $wi(q)$, as introduced by Becher in \cite{B}, is the least such integer $n$ if $q$ is weakly isotropic, and is infinite otherwise. In \cite{BGBM}, Berhuy, Grenier-Boley and Mahmoudi defined the \emph{$q$-sublevel of $F$}, denoted $\sub_q(F)$, to be the least $n\in\mathbb{N}$ such that the orthogonal sum of $n+1$ copies of $q$ is isotropic over $F$, and to be infinite if no such $n$ exists. Clearly, we have that $\sub_q(F)=wi(q)-1$. Throughout this article, our preference will be to write in terms of the $q$-sublevel, bearing in mind that all such results can immediately be reformulated in terms of the weak isotropy index of $q$.

We open our study of the $q$-sublevel with an examination of its relationships to other invariants, before moving on to explore its behaviour with respect to field extensions. In particular, for $q$ a form over $F$, we seek to determine the entries of the set $\{\sub_q(K)\mid K/F \textrm{ field extension}\}$. Drawing upon known results regarding isotropy over function fields of forms, we establish criteria for the containment of integers within this set. Without placing any restrictions on $q$, we can determine some entries of the set $\{\sub_q(K)\mid K/F \textrm{ field extension}\}$. Moreover, restricting to forms $q$ of a specific type (such as forms of $2$-power dimension, forms with maximal splitting, forms of certain signature, etc.), we can identify additional entries of the set. Indeed, for certain forms $q$, we obtain a complete determination of the set $\{\sub_q(K)\mid K/F \textrm{ field extension}\}$. In the complementary direction, without placing any restrictions on the form $q$, we invoke our aforementioned lower bound on the first Witt index of multiples of Pfister forms to show that all the integers contained in certain intervals do not belong to the set $\{\sub_q(K)\mid K/F \textrm{ field extension}\}$ (see Theorem~\ref{inadmissiblesub}). Moreover, we show that these intervals of inadmissible integers cannot be extended in either direction.



Letting $R$ be a non-trivial ring, the \emph{level of $R$}, denoted $s(R)$, is the least positive integer $n$ such that $-1$ is a sum of $n$ squares in $R$ if such an $n$ exists, and is infinite otherwise. Interest in this invariant first arose on account of the Artin-Schreier theorem, which states that a field $F$ has an ordering if and only if $s(F)=\infty$, and its behaviour with respect to various classes of rings continues to be a topic of study. In \cite{BGBM}, Berhuy, Grenier-Boley and Mahmoudi introduced the concept of the level of a field with respect to a form $q$, or the $q$-level, as a generalisation of the level, and undertook a wide-ranging investigation of this invariant. For $q$ a form over $F$, the \emph{$q$-level of $F$}, denoted $s_q(F)$, is the least $n\in\mathbb{N}\setminus \{0\}$ such that the orthogonal sum of $n$ copies of $q$ represents $-1$ over $F$, and is infinite if no such $n$ exists. Clearly, one recovers the level of a field by considering the $q$-level with respect to the form $\< 1\>$. As was the case with the $q$-sublevel, we interest ourselves in the behaviour of this invariant with respect to field extensions, seeking a determination of the set $\{s_q(K)\mid K/F \textrm{ field extension}\}$. In this regard, we establish analogues of our results with respect to the $q$-sublevel, establishing criteria for containment within this set; identifying elements of this set common to all forms; determining additional (and sometimes, all) entries of this set with respect to specific forms; showing that certain integers are not entries of this set with respect to specific forms. Consequently, we can answer some questions raised in \cite{BGBM}. In particular, for all $n\in\mathbb{N}\setminus \{0\}$, in Theorem~\ref{svalues} we establish the existence of an $n$-dimensional form $q$ over an ordered field $F$ that can take any prescribed positive integer as its $q$-level over a suitable extension, answering \cite[Question $6.1$]{BGBM}.

In the final section of this article, the $q$-sublevel and $q$-level are considered with respect to round forms and Pfister forms, classes of forms for which these invariants are known to coincide. Addressing remarks in \cite{BGBM}, we compare and contrast the behaviour of the $q$-sublevel with respect to these classes, showing that certain of their results can be extended from the class of Pfister forms to round forms, whereas others cannot. We conclude with a treatment of the $q$-sublevel with respect to function fields of quadratic forms in the case where $q$ is a Pfister form.

Henceforth, we will let $F$ denote a field of characteristic different from two (indeed, if $\ch(F)=2$ then every anisotropic quadratic form $q$ over $F$ satisfies $\sub_q(F)=1$). The term ``form'' will refer to a regular quadratic form. Every form over $F$ can be diagonalised. Given $a_1,\ldots ,a_n\in F^{\x}$ for $n\in\mathbb{N}\setminus \{ 0\}$, one denotes by $\<a_1,\ldots,a_n\>$ the $n$-dimensional quadratic form $a_1X_1^2+\ldots+a_nX_n^2$. If $p$ and $q$ are forms over $F$, we denote by $p\perp q$ their orthogonal sum and by $p\otimes q$ their tensor product. For $n\in\mathbb{N}$, we will denote the orthogonal sum of $n$ copies of $q$ by $n\x q$. We use $aq$ to denote $\<a\>\otimes q$ for $a\in F^{\x}$. We write $p\simeq q$ to indicate that $p$ and $q$ are isometric, and say that $p$ and $q$ are \emph{similar} (over $F$) if $p\simeq aq$ for some $a\in F^{\x}$. For $q$ a form over $F$ and $K/F$ a field extension, we will view $q$ as a form over $K$ via the canonical embedding. A form $p$ is a \emph{subform of $q$} if $q\simeq p\perp r$ for some form $r$, in which case we will write $p\subset q$. A form $q$ \emph{represents $a\in F$} if there exists a vector $v$ such that $q(v)=a$. We denote by $D_F(q)$ the set of values in $F^{\x}$ represented by $q$. A form over $F$ is \emph{isotropic} if it represents zero non-trivially, and \emph{anisotropic} otherwise. A form $q$ over $F$ is \emph{universal} if $D_F(q)=F^{\x}$. In particular, isotropic forms are universal \cite[Theorem I.$3.4$]{LAM}. Every form $q$ has a decomposition $q\simeq q_{\mathrm{an}}\perp i_W(q)\x\<1,-1\>$ where the anisotropic form $q_{\mathrm{an}}$ and the integer $i_W(q)$ are uniquely determined. A form $q$ is \emph{hyperbolic} if $q_{\mathrm{an}}$ is trivial, whereby $i_W(q)=\frac 1 2 \dim q$. Two anisotropic forms $p$ and $q$ over $F$ are \emph{isotropy equivalent} if for every field extension $K/F$ we have that $p_K$ is isotropic if and only if $q_K$ is isotropic. The following basic fact (see \cite[Exercise I.$16$]{LAM}) will be employed frequently. 
\begin{lemma}\label{star} If $\tau\subset\varphi$ with $\dim\tau\geqs\dim\varphi-i_W(\varphi)+1$, then $\tau$ is isotropic.\end{lemma}



An \emph{ordering of $F$} is a set $P\subset F^{\x}$ such that $P\cup -P=F^{\x}$ and $x+y, xy\in P$ for all $x,y\in P$. We will let $X_F$ denote the space of orderings of $F$. If $X_F$ is non-empty, we say that $F$ is a \emph{formally real field}. An element $a\in F^{\x}$ is \emph{totally positive} if $a\in P$ for all $P\in X_F$, which is the case if and only if $a$ is a sum of squares in $F^{\x}$, denoted $a\in\sum {F^{\x}}^2$. For $a\in\sum {F^{\x}}^2$, the \emph{length of $a$}, $\ell_F(a)$, is the least number of squares in $F^{\x}$ that sum to $a$ (we set $\ell_F(a)=\infty$ if $a\notin \sum {F^{\x}}^2$). The \emph{Pythagoras number of $F$} is $p(F)=\sup\{ \ell_F(a)\mid a\in\sum {F^{\x}}^2\}$. Given a form $q$ over $F$ and an ordering $P\in X_F$, the \emph{signature of $q$ at $P$}, denoted $\sgn_P(q)$, is the number of coefficients in a diagonalisation of $q$ that are in $P$ minus the number that are not in $P$. A form $q$ over $F$ is \emph{indefinite at $P\in X_F$} if $|\sgn_P(q)|<\dim q$. For $F$ a field without orderings, the \emph{$u$-invariant of $F$} is $u(F)=\sup\{\dim q\mid q\textrm{ is an anisotropic form over }F\}$. For $n,m\in\mathbb{N}$ with $m$ positive, we will often invoke the identity $\left\lceil\frac n m\right\rceil=\left\lfloor\frac {n-1} {m}\right\rfloor+1$.

For $n\in\mathbb{N}\setminus \{ 0\}$, an \emph{$n$-fold Pfister form} over $F$ is a form isometric to $\<1,a_1\>\otimes\ldots\otimes\<1,a_n\>$ for some $a_1,\ldots ,a_n\in F^{\times}$ (the form $\< 1\>$ is the $0$-fold Pfister form). Isotropic Pfister forms are hyperbolic \cite[Theorem X.1.7]{LAM}. A form $\tau$ over $F$ is a \emph{neighbour} of a Pfister form $\pi$ if $\tau\subset a\pi$ for some $a\in F^{\x}$ and $\dim{\tau}>\frac 1 2 \dim\pi$. An anisotropic form $q$ is isotropy equivalent to a Pfister form $\pi$ if and only if $q$ is a neighbour of $\pi$ \cite[Proposition 2]{H3}. A form $q$ over $F$ is a \emph{group form} if $D_F(q)$ is a subgroup of $F^{\x}$. A form $q$ over $F$ is \emph{round} if $D_F(q)=\{a\in F^{\x}\mid aq\simeq q\}$, the group of similarity factors of $q$. Pfister forms are round, see \cite[Theorem X.1.8]{LAM}. Indeed, Witt's Round Form Theorem \cite[Theorem X.$1.14$]{LAM} states that the product of a Pfister form and a round form is round.

For a form $q$ over $F$ with $\dim q=n\geqs 2$ and $q\not\simeq\<1,-1\>$, the \emph{function field $F(q)$ of $q$} is the quotient field of the integral domain $F[X_1,\ldots ,X_n]/(q(X_1,\ldots ,X_n))$ (this is the function field of the affine quadric $q(X)=0$ over $F$). To avoid case distinctions, we set $F(q)=F$ if $\dim q\leqs 1$ or $q\simeq\<1,-1\>$. The integer $i_W(q_{F(q)})$ (which is positive for all forms $q$ of dimension greater than one) is called the \emph{first Witt index of $q$}, and is denoted by $i_1(q)$. For all forms $p$ over $F$ and all extensions $K/F$ such that $q_K$ is isotropic, we have that $i_W(p_{F(q)})\leqs i_W(p_K)$ (see \cite[Proposition 3.1 and Theorem 3.3]{Kn1}). In particular, we have that $i_1(q)\leqs i_W(q_K)$ for all extensions $K/F$ such that $q_K$ is isotropic. An anisotropic form $q$ of dimension $2^n+m$, where $n\geqs 0$ and $1\leqs m\leqs 2^n$, is said to have \emph{maximal splitting} if $i_1(q)=m$. As per \cite[Theorem X.$4.1$]{LAM}, $F(q)$ is a purely-transcendental extension of $F$ if and only if $q$ is isotropic over $F$. On account of this fact, one can see that two anisotropic forms $p$ and $q$ over $F$ are \emph{isotropy equivalent} if and only if $p_{F(q)}$ and $q_{F(p)}$ are isotropic (if $p_{F(q)}$ is isotropic, then $p_{K(q)}$ is isotropic for all $K/F$ such that $q_K$ is isotropic, whereby $p_{K}$ is isotropic since $K(q)/K$ is a purely-transcendental extension). The behaviour of orderings with respect to function field extensions is governed by the following result due to Elman, Lam and Wadsworth \cite[Theorem 3.5]{ELW} and, independently, Knebusch \cite[Lemma 10]{GS}.
\begin{theorem}\label{ELW} Let $q$ be a form over a formally real field $F$ such that $\dim q\geqs 2$. Then $P\in X_F$ extends to $F(q)$ if and only if $q$ is indefinite at $P$.
\end{theorem}

%
%


\cite[Theorem 1]{H3} and \cite[Theorem 4.1]{KM} represent important isotropy criteria for function fields of quadratic forms. We will regularly invoke these results throughout this article, and therefore recall them below.


\begin{theorem}\label{H95} $($Hoffmann$)$ Let $p$ and $q$ be forms over $F$ such that $p$ is anisotropic. If $\dim p\leqs 2^n<\dim q$ for some $n\in\mathbb{N}$, then
$p_{F(q)}$ is anisotropic.\end{theorem}

%
%

\begin{theorem}\label{km} $($Karpenko, Merkurjev$)$ Let $p$ and $q$ be anisotropic forms over $F$ such that $p_{F(q)}$ is isotropic. Then \begin{enumerate}[$(i)$]
\item $\dim p - i_1(p)\geqs\dim q - i_1(q)$;
\vspace{.1cm}
\item $\dim p - i_1(p)=\dim q - i_1(q)$ if and only if $q_{F(p)}$ is isotropic.\end{enumerate}\end{theorem}



\section{The isotropy of multiples of Pfister forms}

Since the isotropy of scalar multiples of Pfister forms is well understood (indeed, an anisotropic form $q$ of dimension at least two is a scalar multiple of a Pfister form if and only if $q$ is hyperbolic over $F(q)$ \cite[Corollary $23.4$]{EKM}), we will restrict our attention to products of Pfister forms with forms of dimension at least two. In his thesis \cite[Th\'eor\`eme 6.4.2]{R}, Roussey established the following.


\begin{theorem}\label{R} $($Roussey$)$ Let $p$ and $q$ be two forms over F of dimension at least two and let $\pi$ be similar to a Pfister form over $F$. If $p$ is isotropic over $F(q)$, then $\pi\otimes p$ is isotropic over $F(\pi\otimes q)$. 
\end{theorem}


In a similar vein to the above, we note that the corresponding result with respect to hyperbolicity also holds, having been established by Fitzgerald \cite[Theorem $3.2$]{F}.

The following example, communicated to me by Karim Becher, demonstrates that the converse of the above theorem does not hold in general.

\begin{example}\label{nocon} Let $q\simeq\< 1,1,1,7\>$ and $\pi\simeq\< 1,1,1,1\>$ over $F=\mathbb{Q}$. Since $\det q\notin \mathbb{F}^2$, the form $q$ is not similar to a $2$-fold Pfister form. Thus, $i_1(q)=1$ by \cite[Corollary $23.4$]{EKM}. Hence, Theorem~\ref{km}~$(i)$ implies that $\< 1,1,1\>$ is anisotropic over $F(q)$. As $7\in D_{F}(\pi)$, we have that $7\< 1,1,1,1\>\simeq\< 1,1,1,1\>$, and thus that $\pi\otimes q\simeq 16\x\< 1\>$. Since $\pi\otimes\< 1,1,1\>$ is a Pfister neighbour of $16\x\< 1\>$, we have that $\pi\otimes\< 1,1,1\>$ is isotropic over $F(\pi\otimes q)$.
\end{example}

While the converse of Theorem~\ref{R} does not generally hold, we can establish it in a restricted setting.

\begin{prop}\label{funnycor} Let $p$ and $q$ be forms over $F$ of dimension at least two such that $\pi\otimes p$ and $\pi\otimes q$ are anisotropic over $F$, where $\pi$ is similar to a Pfister form over $F$. Suppose that $q$ has maximal splitting and that $q$ is isotropic over $F(p)$. If $\pi\otimes p$ is isotropic over $F(\pi\otimes q)$, then $p$ is isotropic over $F(q)$. 
\end{prop}

\begin{proof} Since $q$ has maximal splitting, it follows that $\dim q-i_1(q)=2^n$ for some $n\in\mathbb{N}$. As $\pi\otimes p$ is isotropic over $F(\pi\otimes q)$, Theorem~\ref{H95} enables us to conclude that $\dim p>2^n$. Since $q$ is isotropic over $F(p)$, Theorem~\ref{km}~$(i)$ implies that $\dim p-i_1(p)=2^n$, whereby Theorem~\ref{km}~$(ii)$ implies that $p$ is isotropic over $F(q)$. 
\end{proof}

The following result regarding the isotropy of multiples of round forms is well known (see \cite[Theorem 2]{WS}).


\begin{theorem}\label{WS} $($Wadsworth, Shapiro$)$ Let $\alpha$ be an anisotropic round form over F and let $q$ be another
form over $F$. If $\alpha\otimes q$ is isotropic, then there exist forms $q_1$ and $q_2$ over $F$ such that $\alpha\otimes q_1$ is anisotropic, $q_2$ is hyperbolic,
and $\alpha\otimes q\simeq \alpha\otimes q_1\perp \alpha\otimes q_2$. In particular, $i_W(\alpha\otimes q)=(\dim\alpha) i_W(q_2)$.
\end{theorem}


%
%
%


With respect to the above theorem, we clearly have that $i_W(q_2)\geqs i_W(q)$. These quantities do not appear to satisfy any stronger relation however (indeed, the form $q$ may be anisotropic). In the case where $q$ is a form over $F$ of dimension at least two, one can apply Theorem~\ref{WS} to a form $\pi$ that is similar to an anisotropic Pfister form over $F$ to obtain that $i_W(\pi\otimes q)\geqs \dim\pi$ in the case where $\pi\otimes q$ is isotropic (see \cite[Theorem $1.4$]{EL} for a related result). Hence, for $q$ and $\pi$ forms over $F$ such that the dimension of $q$ is at least two, $\pi$ is similar to a Pfister form and $\pi\otimes q$ is anisotropic over $F$, we have that $i_1(\pi\otimes q)\geqs \dim\pi$. Invoking Theorem~\ref{R}, we can improve this lower bound on the first Witt index of a multiple of a Pfister form.


\begin{theorem}\label{i1bound} Let $q$ a form over $F$ of dimension at least two and $\pi$ similar to a Pfister form over $F$ be such that $\pi\otimes q$ is anisotropic over $F$. Then $i_1(\pi\otimes q)\geqs (\dim\pi)i_1(q)$.
\end{theorem}


\begin{proof} If $i_1(q)=1$, then the result follows from invoking Theorem~\ref{WS} with respect to the field $F(\pi\otimes q)$ (we note that $\pi$ is anisotropic over $F(\pi\otimes q)$ by Theorem~\ref{H95}). Otherwise, let $q'\subset q$ over $F$ of dimension $\dim q -i_1(q)+1$. Lemma~\ref{star} implies that $q'$ is isotropic over ${F(q)}$. Hence, $\pi\otimes q'$ is isotropic over $F(\pi\otimes q)$ by Theorem~\ref{R}. As $\pi\otimes q'\subset\pi\otimes q$, we have that $\pi\otimes q'$ is anisotropic over $F$ and, furthermore, that $\pi\otimes q$ is isotropic over $F(\pi\otimes q')$, whereby $\pi\otimes q'$ and $\pi\otimes q$ are isotropy equivalent. Invoking Theorem~\ref{km}~$(i)$, we have that $\dim(\pi\otimes q')-i_1(\pi\otimes q')= \dim(\pi\otimes q)-i_1(\pi\otimes q)$, whereby $i_1(\pi\otimes q)=i_1(\pi\otimes q')+\dim\pi (\dim q -\dim q')=i_1(\pi\otimes q')+\dim\pi (i_1(q)- 1)$. Since $i_1(\pi\otimes q')\geqs\dim\pi$ by Theorem~\ref{WS}, we have that $i_1(\pi\otimes q)\geqs (\dim\pi)i_1(q)$.
\end{proof}



With respect to the above theorem, the following example shows that $i_1(\pi\otimes q)$ can exceed $(\dim\pi)i_1(q)$. 

\begin{example} As in Example~\ref{nocon}, let $q\simeq\< 1,1,1,7\>$ and $\pi\simeq\< 1,1,1,1\>$ over $F=\mathbb{Q}$. As before, $i_1(q)=1$ and $\pi\otimes q$ is isometric to the Pfister form $16\x\< 1\>$. Hence, we have that $i_1(\pi\otimes q)=8 >(\dim\pi)i_1(q)=4$.
\end{example}


%

For certain forms $q$ over $F$, we can determine the value of $i_1(\pi\otimes q)$.

\begin{prop}\label{i1maxspl} Let $q$ a form over $F$ of dimension at least two and $\pi$ similar to a Pfister form over $F$ be such that $\pi\otimes q$ is anisotropic over $F$. If $q$ has maximal splitting, then $i_1(\pi\otimes q)= (\dim\pi)i_1(q)$. 
\end{prop}

\begin{proof} Let $\dim q= 2^n+k$ for some $n,k\in\mathbb{N}$, where $0< k\leqs 2^n$. Hence, $\dim(\pi\otimes q)=2^n\dim\pi+k\dim\pi$, where $0< k\dim\pi\leqs 2^n\dim\pi$. As $i_1(q)=k$, 
Theorem~\ref{i1bound} implies that $i_1(\pi\otimes q)\geqs k\dim\pi$. Let $\vartheta\subset\pi\otimes q$ over $F$ such that $\dim\vartheta =2^n\dim\pi$. If $i_1(\pi\otimes q)> k\dim\pi$, then Lemma~\ref{star} implies that $\vartheta$ is isotropic over $F(\pi\otimes q)$, contradicting Theorem~\ref{H95}. Thus, $i_1(\pi\otimes q)= (\dim\pi)i_1(q)$.
\end{proof}

\begin{cor}\label{maxsplcor} Let $q$ a form over $F$ of dimension at least two and $\pi$ similar to a Pfister form over $F$ be such that $\pi\otimes q$ is anisotropic over $F$. If $q$ has maximal splitting, then $\pi\otimes q$ has maximal splitting.
\end{cor}

\begin{proof} As follows from the definition, an anisotropic form over $F$ of dimension at least two has maximal splitting if and only if its dimension minus its first Witt index is a $2$-power. By Proposition~\ref{i1maxspl}, we have that $\dim(\pi\otimes q)-i_1(\pi\otimes q)=\dim\pi(\dim q-i_1(q))$. Since $q$ has maximal splitting, $\dim q-i_1(q)=2^k$ for some $k\in\mathbb{N}$, whereby $\dim(\pi\otimes q)-i_1(\pi\otimes q)=2^{n+k}$ for some $n\in\mathbb{N}$. Hence, $\pi\otimes q$ has maximal splitting.
\end{proof}

\section{Basic properties of the weak isotropy index}

By definition, the $q$-sublevel and $q$-level of $F$ satisfy $$\sub_q(F)=\inf\{ n\in\mathbb{N}\mid (n+1)\x q\text{ is isotropic over }F\}$$ and $$s_q(F)=\inf\{ n\in\mathbb{N}\mid \< 1\>\perp n\x q\text{ is isotropic over }F\}.$$ An important distinction between these concepts is the fact that the $q$-sublevel is invariant under scaling, whereas the $q$-level is generally not. 

If $q$ is an isotropic (and hence universal) form over $F$, then $\sub_q(F)=0$ and $s_q(F)=1$. Thus, we will restrict our attention to forms $q$ that are anisotropic over $F$. Clearly, if $\sub_q(F)<\infty$, then the form $(\sub_q(F)+1)\x q$ is isotropic over $F$, whereby we have that $s_q(F)\leqs \sub_q(F)+1$ (as per \cite[Lemma $3.1~(7)$]{BGBM}).

Our opening result records some basic properties of the $q$-sublevel of a field, by establishing analogues of statements in \cite[Lemma 3.1 and Proposition 3.3]{BGBM} concerning the $q$-level.

\begin{prop}\label{basicsub} Let $q$ be an anisotropic form over $F$.

\begin{enumerate}[$(i)$] \item $1\leqs \sub_q(F)\leqs s(F)$.
\vspace{.1cm}

\item If $q'\subset q$, then $\sub_q(F)\leqs \sub_{q'}(F)$.
\vspace{.1cm}

\item If $K/F$ is a field extension, then $\sub_q(K)\leqs\sub_{q}(F)$.
\vspace{.1cm}

\item If $K/F$ is a field extension whose degree is odd, then $\sub_q(K)=\sub_{q}(F)$.
\vspace{.1cm}

\item $\sub_q(F)=\sub_{q}(F(x))$, and $\sub_q(F)=\sub_{q\perp\< x\>}(F(x))$ if $F$ is formally real.
\vspace{.1cm}

\item For every $n\in\mathbb{N}\setminus \{ 0\}$, we have that $\sub_{n\x q}(F)=\left\lfloor \frac {\sub_q(F)} n\right\rfloor$.
\vspace{.1cm}

\item If $\sub_q(F)<\infty$, then $\sub_q(F)\leqs p(F)-1$.
\vspace{.1cm}

\item If $F$ is not formally real, then $\sub_q(F)\leqs\left\lceil\frac {u(F)} {\dim q}\right\rceil\leqs u(F)$.
\end{enumerate} 
\end{prop}


\begin{proof} $(i)$, $(ii)$ and $(iii)$ easily follow from the definition of the $q$-sublevel of a field, while $(iv)$ can be proven by invoking Springer's Theorem \cite[Theorem VII.$2.7$]{LAM}.

To prove $(v)$, we recall that every anisotropic form over $F$ remains anisotropic over $F(x)$ (see \cite[Lemma IX.$1.1$]{LAM}). Moreover, if $F$ is formally real and $n\x(q\perp\< x\>)$ is isotropic over $F(x)$ for some $n\in\mathbb{N}$, then \cite[Exercise IX.$1$]{LAM} implies that $n\x q$ is isotropic over $F$.



To prove $(vi)$, we note that $\left(\left\lceil\frac {\sub_q(F)+1} {n}\right\rceil\right)n\x q$ is isotropic, whereby $\sub_{n\x q}(F)\leqs \left\lceil\frac {\sub_q(F)+1} {n}\right\rceil -1$. Since $\left\lceil\frac {\sub_q(F)+1} {n}\right\rceil-1=\left\lfloor \frac {\sub_q(F)} n +1\right\rfloor-1=\left\lfloor \frac {\sub_q(F)} n\right\rfloor$, we have that $\left(\left\lceil\frac {\sub_q(F)+1} {n}\right\rceil-1\right)n\x q$ is anisotropic, establishing $(vi)$.


To prove $(vii)$, we may assume that $p(F)<\infty$. Since $(\sub_q(F) +1)\x q$ is isotropic, we have that $p(F)\x q$ is isotropic. Hence, $\sub_q(F)\leqs p(F)-1$.


Statement $(viii)$ follows from the fact that $\left(\left\lceil\frac {u(F)} {\dim q}\right\rceil +1\right)\x q$ is isotropic.
\end{proof}

\begin{remark} We remark that all of the bounds in Proposition~\ref{basicsub} can be attained.

As $\sub_{\< 1\>}(F)=s(F)$, letting $q\simeq\< 1\>$ over a field $F$ that is not formally real, one realises the upper bound in $(i)$.

Invoking $(v)$, one sees that the upper bound in $(ii)$ can be attained in the case where $q'$ is a proper subform of $q$.


The attainability of the upper bound in $(iii)$ can be deduced from $(iv)$ or $(v)$.

The upper bound in $(vii)$ can be realised by letting $q\simeq\< 1\>$ over a field $F$ of finite Pythagoras number satisfying $p(F)=s(F)+1$ (see \cite[Ch. 7, Proposition 1.5]{P}).

Finally, as per \cite[Remark 3.4]{BGBM}, letting $q\simeq\< 1\>$ over a field $F$ such that $s(F)=u(F)=2^m$, one realises the upper bounds in $(viii)$.
\end{remark}


%
%
%



As was observed in \cite[Lemma $3.1~(8)$]{BGBM}, if $q$ is an anisotropic form over $F$ such that $1\in D_F(q)$, then $\sub_q(F)\leqs s_q(F)$, since $\< 1\>\perp s_q(F)\x q\subset (s_q(F)+1)\x q$ in this case. Indeed, more generally, we have the following relation.

\begin{prop}\label{ssubrel} Let $q$ be an anisotropic form over $F$ and $a\in F^{\x}$. Then $$\sub_q(F)=\inf\{s_{aq}(F)\mid a\in D_F(q)\}.$$
\end{prop}

\begin{proof} By definition, $\sub_q(F)=\inf\{ n\in\mathbb{N}\mid (n+1)\x q\text{ is isotropic over }F\}$.

As $(n+1)\x q$ is isotropic over $F$ if and only if there exists $a\in F^{\x}$ such that $a\in D_F(q)$ and $-a\in D_F(n\x q)$, we can conclude that $$\sub_q(F)=\inf\{ n\in\mathbb{N}\mid -a\in D_F(n\x q)\text{ for some }a\in D_F(q)\}.$$ Hence, we have that $\sub_q(F)=\inf\{ n\in\mathbb{N}\mid -1\in D_F(n\x aq)\text{ for some }a\in D_F(q)\}$. Hence, we can conclude that $\sub_q(F)=\inf\{s_{aq}(F)\mid a\in D_F(q)\}$.
\end{proof}

For $q$ a form over $F$ and $a\in F^{\x}$, $\ell_q(a)=\inf\{ n\in\mathbb{N}\mid n\x q\perp \< -a\>\textrm{ is isotropic over }F\}$ corresponds to the \emph{$q$-length of $a\in F^{\x}$}, as introduced in \cite{BGBM}. The \emph{Pythagoras $q$-number of $F$} is $p_q(F)=\sup\{\ell_q(a)\mid a\in F^{\x}\textrm{ has }\ell_q(a)<\infty\}$. 

As above, for $q$ an anisotropic form over $F$, the finiteness of $\sub_q(F)$ implies that of $s_q(F)$. Indeed, we have the following result.



\begin{prop}\label{sandsub} For $q$ an anisotropic form over $F$, the following are equivalent:
\begin{enumerate}[$(i)$]
\item $\sub_q(F)<\infty$.
\vspace{.1cm}

\item $s_q(F)<\infty$ and $s_{-q}(F)<\infty$.
\vspace{.1cm}

\item $p_q(F)<\infty$ and $p_q(F)\x q$ is universal. 
\end{enumerate}
\end{prop}


\begin{proof} Assuming $(i)$, we have that $(\sub_q(F)+1)\x q$ is isotropic, and thus universal, whereby $p_q(F)\x q$ is universal and $p_q(F)\leqs \sub_q(F)+1$, establishing $(iii)$. 

Assuming $(iii)$, we have that $\{-1,1\}\subset D_F(p_q(F)\x q)$, whereby $s_q(F)\leqs p_q(F)$ and $s_{-q}(F)\leqs p_q(F)$, establishing $(ii)$.

Assuming $(ii)$, we have that $(s_q(F)+s_{-q}(F))\x q$ is isotropic, whereby $\sub_q(F)\leqs s_q(F)+s_{-q}(F)-1$, establishing $(i)$.
\end{proof}


With respect to the above result, we note the existence of forms $q$ over fields $F$ such that $s_q(F)<\infty$ and $\sub_q(F)=\infty$ (see Remark~\ref{remsandsub}).

\begin{prop}\label{subpq} Let $q$ be an anisotropic form over $F$ such that $\sub_q(F)<\infty$. Then $p_q(F)-1\leqs \sub_q(F)\leqs p_q(F)$.
\end{prop}

\begin{proof} As per the above proof, $p_q(F)-1\leqs \sub_q(F)$. Moreover, since $p_q(F)\x q$ is universal, we have that $(p_q(F)+1)\x q$ is isotropic, whereby $\sub_q(F)\leqs p_q(F)$.
\end{proof}

\begin{remark} Letting $q\simeq \< 1\>$ over a field $F$ such that $\sub_q(F)<\infty$, Proposition~\ref{subpq} states that $p(F)-1\leqs s(F)\leqs p(F)$. As per \cite[Ch. 7, Proposition 1.5]{P}, there exist fields $K$ and $L$ satisfying $s(K)=p(K)-1<\infty$ and $s(L)=p(L)<\infty$.  
\end{remark}


%
%
%
%
%

\section{Values of the weak isotropy index}


In this section, we study the behaviour of the $q$-sublevel (or equivalently, the weak isotropy index) with respect to field extensions. In particular, for $q$ an anisotropic form over $F$, we will study the set $\{\sub_q(K)\mid K/F \textrm{ field extension}\}$. Clearly, $\sub_q(F)$ always belongs to this set, with the remaining entries being less than $\sub_q(F)$. 

%
%

We begin by seeking to show that certain prescribed integers belong to the above set. As motivated earlier, function fields of associated quadratic forms are the natural field extensions to consider in this regard. Invoking Theorem~\ref{H95} to this end, we can record the following result.

\begin{prop}\label{hsub} Let $q$ be an anisotropic form over $F$. If an integer $m\leqs \sub_q(F)$ is such that $m\dim q\leqs 2^n<(m+1)\dim q$ for some $n\in\mathbb{N}$, then we have that $m\in\{\sub_q(K)\mid K/F \textrm{ field extension}\}$.
\end{prop}

\begin{proof} If $m=\sub_q(F)$, then there is nothing to prove. If $m<\sub_q(F)$, then $(m+1)\x q$ is anisotropic over $F$. Letting $K=F((m+1)\x q)$, Theorem~\ref{H95} implies that $m\x q$ is anisotropic over $K$, whereby $\sub_q(K)=m$.
\end{proof}

We next determine the integers to which the above result can be applied.

\begin{prop}\label{subfigure} Let $q$ be an anisotropic form over $F$. An integer $m\leqs \sub_q(F)$ is such that $m\dim q\leqs 2^n<(m+1)\dim q$ if and only if $m=\left\lfloor\frac {2^n} {\dim q}\right\rfloor$ for some $n\in\mathbb{N}$. In particular, $m=\left\lfloor\frac {2^n} {\dim q}\right\rfloor\in\{\sub_q(K)\mid K/F \textrm{ field extension}\}$.
\end{prop}

\begin{proof}  Let $m\leqs \sub_q(F)$ be such that $m\dim q\leqs 2^n<(m+1)\dim q$ for some $n\in\mathbb{N}$. Since $m\dim q\leqs 2^n$, it follows that $m\leqs \frac {2^n} {\dim q}$, and thus $m\leqs\left\lfloor\frac {2^n} {\dim q}\right\rfloor$ as $m\in\mathbb{N}$. Moreover, as $2^n<(m+1)\dim q$, we have that $m>\frac {2^n} {\dim q}-1$. Hence, we have that $m\geqs \left\lfloor\frac {2^n} {\dim q}\right\rfloor$, and thus we can conclude that $m=\left\lfloor\frac {2^n} {\dim q}\right\rfloor$.

Conversely, letting $m=\left\lfloor\frac {2^n} {\dim q}\right\rfloor$, we clearly have that $\dim\left(m\x q\right)\leqs 2^n$. Moreover, as $\left\lfloor\frac {2^n} {\dim q}\right\rfloor =\left\lceil \frac {2^n+1} {\dim q}\right\rceil-1$, letting $m=\left\lfloor\frac {2^n} {\dim q}\right\rfloor$ gives us that $\dim\left(\left(m+1\right)\x q\right)> 2^n$. 


The last statement follows from applying Proposition~\ref{hsub} to $m=\left\lfloor\frac {2^n} {\dim q}\right\rfloor$.
\end{proof}

\begin{cor}\label{lowsub} Let $q$ be an anisotropic form over $F$ of dimension at least two. Then $\{ 0,1\}\subseteq\{\sub_q(K)\mid K/F \textrm{ field extension}\}$.
\end{cor}

\begin{proof} Since $q_{F(q)}$ is isotropic, one has that $0\in\{\sub_q(K)\mid K/F \textrm{ field extension}\}$.

For $\dim q= 2^{n-1}+k$ for some $n,k\in\mathbb{N}$ where $0< k \leqs 2^{n-1}$, Proposition~\ref{subfigure} implies that $\left\lfloor\frac {2^n} {\dim q}\right\rfloor=1\in\{s_q(K)\mid K/F \textrm{ field extension}\}$.
\end{proof}




\begin{cor}\label{2powersub} Let $q$ be an anisotropic form over $F$ of dimension $2^n$ for some $n\in\mathbb{N}$. Then $2^k\in\{\sub_q(K)\mid K/F \textrm{ field extension}\}$ for all integers $k\geqs 0$ such that $2^k\leqs \sub_q(F)$.
\end{cor}

\begin{proof} This follows immediately from Proposition~\ref{hsub} or Proposition~\ref{subfigure}.
\end{proof}

With respect to certain forms $q$ and integers $m$, Proposition~\ref{subfigure} enables us to determine whether or not $m$ is in $\{\sub_q(K)\mid K/F\text{ field extension}\}$.


\begin{prop}\label{submodmax} Let $q$ be a form over $F$ such that $(m+1)\x q$ has maximal splitting for some integer $m<\sub_q(F)$. Then $m\in\{s_q(K)\mid K/F\text{ field extension}\}$ if and only if $m=\left\lfloor\frac {2^n} {\dim q}\right\rfloor$ for some $n\in\mathbb{N}$.
\end{prop} 

\begin{proof} For $m=\left\lfloor\frac {2^n} {\dim q}\right\rfloor$ for some $n\in\mathbb{N}$, Proposition~\ref{subfigure} implies that we have $m\in\{\sub_q(K)\mid K/F\text{ field extension}\}$. 

Conversely, let $K/F$ be such that $\sub_q(K)=m$. As $m\x q$ is anisotropic over $K$, Lemma~\ref{star} implies that $\dim(m\x q)\leqs \dim((m+1)\x q)-i_W({(m+1)\x q}_K)$. We recall that $i_W({(m+1)\x q}_K)\geqs i_1((m+1)\x q)$, as the form $(m+1)\x q$ is isotropic over $K$. Thus, we have that $\dim(m\x q)\leqs \dim((m+1)\x q)-i_1((m+1)\x q)$. Since $(m+1)\x q$ has maximal splitting, $\dim((m+1)\x q)-i_1((m+1)\x q)= 2^n$ for some $n\in\mathbb{N}$. Hence, it follows that $\dim(m\x q)\leqs 2^n <  \dim((m+1)\x q)$. Invoking Proposition~\ref{subfigure}, we have that $m=\left\lfloor\frac {2^n} {\dim q}\right\rfloor$.
\end{proof}

If $q$ is similar to an $n$-fold Pfister form over $F$ for $n\geqs 0$, then for all $m\in\mathbb{N}$ we have that $(m+1)\x q$ is a neighbour of a Pfister form similar to $2^k\x q$, where $2^{k-1}< m+1\leqs 2^k$. Thus, for all $m\in\mathbb{N}$, the form $(m+1)\x q$ has maximal splitting, whereby Proposition~\ref{submodmax} can be applied to all integers $m<\sub_q(F)$ in this case (see Remark~\ref{6.1rec}).

Theorem~\ref{km}~$(i)$ enables us to record another criterion for the admissibility of a prescribed integer in the set $\{\sub_q(K)\mid K/F \textrm{ field extension}\}$.

\begin{prop}\label{KMusagesub} Let $q$ be an anisotropic form over $F$ and $n\in\mathbb{N}$. If $(n+1)\x q$ is anisotropic over $F$ and $\dim q > i_1((n+1)\x q) - i_1(n\x q)$, then $\sub_q(F((n+1)\x q))=n$.
\end{prop}  

\begin{proof} Given our assumptions, Theorem~\ref{km}~$(i)$ allows us to conclude that $n\x q$ is anisotropic over $F((n+1)\x q)$, whereby the result follows.
\end{proof}

Although there exists an explicit determination of the possible values of the first Witt index of a given form in terms of its dimension (see \cite{K}), pinpointing the exact value taken by this invariant remains problematic. Thus, the above criterion is difficult to apply in general. In the case where the function field of a given form has an ordering however, the signature of the form with respect to this ordering imposes a natural bound on the first Witt index of the form, and those of its multiples. Hence, for certain forms $q$ over certain fields $F$, we will apply Proposition~\ref{KMusagesub}  to determine further entries of $\{\sub_q(K)\mid K/F \textrm{ field extension}\}$.



\begin{theorem}\label{subvalues} Let $q$ be an anisotropic form over a formally real field $F$ such that $| \sgn_P(q)| =\dim q -2$ for some $P\in X_F$. Then, for $m=\min\{\dim q -1, \sub_q(F) -1\}$, we have that $\{0,\ldots ,m\}\cup\{\sub_q(F)\}\subseteq\{\sub_q(K)\mid K/F \textrm{ field extension}\}$.
\end{theorem}

\begin{proof} As above, $\sub_q(F)\in\{\sub_q(K)\mid K/F \textrm{ field extension}\}$. 

Let $n\in\mathbb{N}$ such that $n\leqs \min\{\dim q -1, \sub_q(F)-1\}$. As $n\leqs \sub_q(F)-1$, we have that $(n+1)\x q$ is anisotropic over $F$. Let $K=F((n+1)\x q)$. Since $q$ is indefinite with respect to $P$, Theorem~\ref{ELW} implies that $P$ extends to $K$. Moreover, as $|\sgn_P((n+1)\x q)|=(n+1)\dim q -2(n+1)$, we have that $i_1((n+1)\x q)\leqs n+1$. As $n+1\leqs\dim q$, we have that $i_1((n+1)\x q)<\dim q + i_1(n\x q)$, whereby Proposition~\ref{KMusagesub} implies that $\sub_q(K)=n$.
\end{proof}


If $\sub_q(F)\leqs \dim q$ for $q$ as above, our determination of $\{\sub_q(K)\mid K/F \textrm{ field extension}\}$ is complete. 

\begin{cor}\label{subvaluescor} Let $q$ be an anisotropic form over a formally real field $F$ such that $| \sgn_P(q)| =\dim q -2$ for some $P\in X_F$. If $\sub_q(F)\leqs \dim q$, then we have that $\{0,\ldots ,\sub_q(F)\}=\{\sub_q(K)\mid K/F \textrm{ field extension}\}$.
\end{cor}

\begin{proof} Since $\sub_q(F)\leqs \dim q$, the result follows from invoking Theorem~\ref{subvalues}.
\end{proof}


For the forms $q$ treated in Theorem~\ref{subvalues}, the following example demonstrates that we do not necessarily have set equality when $\sub_q(F)>\dim q$.

\begin{example} Let $q$ be a $4$-dimensional form over a formally real field $F$ such that $| \sgn_P(q)| =2$ for some $P\in X_F$ and $\sub_q(F)\geqs 5$ (for example, for $F_0$ a formally real field, one can let $F=F_0(X_1,X_2,X_3,X_4)$ and $q\simeq \< X_1,X_2,X_3,X_4\>$, whereby we have that $\sub_q(F)=\infty$ and $| \sgn_P(q)| =2$ for $P\in X_F$ such that $\{X_1,-X_2,-X_3,-X_4\}\subset P$). By Theorem~\ref{subvalues}, we have that $\{ 0,1,2,3\}\cup\{ \sub_q(F)\}\subseteq \{\sub_{q}(K)\mid K/F \textrm{ field extension}\}$. For $K=F(5\x q)$, it follows from Theorem~\ref{H95} that $4\x q$ is anisotropic over $K$, whereby $\sub_q(K)=4$. Hence, we can conclude that $\{ 0,1,2,3\}\cup\{ \sub_q(F)\}\subsetneq \{\sub_{q}(K)\mid K/F \textrm{ field extension}\}$.
\end{example}

While equality in Theorem~\ref{subvalues} does not generally hold in the case where $\sub_q(F)>\dim q$, it can occur.


\begin{example} Let $q$ be a $3$-dimensional form over a formally real field $F$ such that $q$ is indefinite at $P\in X_F$ and $\sub_q(F)=4$. This can be achieved, for example, by letting $F=F_0(X)$ and $q\simeq \< 1,-a,X\>$ for $F_0$ a formally real field such that $p(F_0)\geqs 5$ and $a\in {F_0}^{\x}$ such that $\ell_{F_0}(a)=5$ (whereby $\< 1\>\perp 4\x \< -a\>$ is anisotropic over $F_0$, implying that its associated Pfister form $4\x\< 1,-a\>$ is anisotropic over $F_0$, and thus that $4\x q$ is anisotropic over $F$ by \cite[Exercise IX.$1$]{LAM}). By Theorem~\ref{subvalues}, $\{ 0,1,2,4\}\subseteq \{\sub_{q}(K)\mid K/F \textrm{ field extension}\}$. Let $K$ be any extension of $F$ such that $4\x q$ is isotropic over $K$. Since $4\x q\simeq (4\x\< 1\>)\otimes q$, we have that $i_W(({4\x q})_{K})\geqs 4$ by Theorem~\ref{WS}. Hence, $3\x q$ is isotropic over $K$ by Lemma~\ref{star}, whereby $\sub_q(K)\leqs 2$. Thus, $\{\sub_q(K)\mid K/F \textrm{ field extension}\}=\{ 0,1,2,4\}$.
\end{example}


For forms $q$ with maximal splitting, we can apply our determination of the first Witt index of their product with Pfister forms, Proposition~\ref{i1maxspl}, to establish the containment of additional integers within the set $\{\sub_q(K)\mid K/F \textrm{ field extension}\}$.





\begin{prop}\label{maxsub} Let $q$ be an anisotropic form over $F$ with maximal splitting. Then $\left\lfloor 2^n-\frac {2^ni_1(q)} {\dim q}\right\rfloor\in \{\sub_q(K)\mid K/F \textrm{ field extension}\}$ for all $n\in\mathbb{N}$ such that $2^n\leqs \sub_q(F)$.
\end{prop}

%

\begin{proof} For convenience, let $m=\left\lfloor 2^n-\frac {2^ni_1(q)} {\dim q}\right\rfloor$, whereby $m\leqs 2^n-1$ and thus $(m+1)\subset 2^n\x q$. We note that $2^n\x q$ is anisotropic over $F$, as $\sub_q(F)\geqs 2^n$. Consider the field extension $K=F(2^n\x q)$. Proposition~\ref{i1maxspl} implies that $i_W({2^n\x q}_K)=i_1(2^n\x q)= 2^ni_1(q)$. As $m+1>2^n-\frac {2^ni_1(q)} {\dim q}$, we have that $\dim((m+1)\x q)> 2^n\dim q-2^ni_1(q)$, whereby Lemma~\ref{star} implies that $(m+1)\x q$ is isotropic over $K$. However, $\dim(m\x q)\leqs 2^n\dim q-2^ni_1(q)$, whereby Theorem~\ref{km}~$(i)$ implies that $m\x q$ is anisotropic over $K$. Thus, $\sub_q(K)=m$.
\end{proof}

For an arbitrary form $q$ over $F$, in order to establish that certain integers less than $\sub_q(F)$ do not belong to the set $\{\sub_q(K)\mid K/F \textrm{ field extension}\}$, function fields of associated quadratic forms are once again the appropriate extensions to consider. 

For $q$ an anisotropic form over $F$, let $n\in\mathbb{N}$ be such that $2^n\leqs\sub_q(F)$. Corollary~\ref{2powersub} states that $2^n\in\{\sub_q(K)\mid K/F \textrm{ field extension}\}$ whenever the dimension of $q$ is a $2$-power. Moreover, $\left\lfloor 2^n-\frac {2^ni_1(q)} {\dim q}\right\rfloor\in \{\sub_q(K)\mid K/F \textrm{ field extension}\}$ for all forms $q$ with maximal splitting, by Proposition~\ref{maxsub}. Interestingly, our next result suggests that these classes of forms represent extreme cases, as we can invoke our lower bound on the first Witt index of multiples of Pfister forms, Theorem~\ref{i1bound}, to show that there do not exist any forms $q$ whose sublevel lies between these integers.






%
%

\begin{theorem}\label{inadmissiblesub} Let $q$ be an anisotropic form over F. If $m\in\left(2^n-\frac {2^ni_1(q)} {\dim q},2^n\right)$ for some $n\in\mathbb{N}$ such that $2^n\leqs \sub_q(F)$, then $m\notin \{\sub_{q}(K)\mid K/F \textrm{ field extension}\}$.
\end{theorem}

\begin{proof} Since $2^n\leqs \sub_q(F)$, we have that $2^n\x q$ is anisotropic over $F$. Moreover, as $m< 2^n$, we have that $(m+1)\x q\subset 2^n\x q$. Let $K/F$ be such that $(m+1)\x q$ is isotropic over $K$. As $2^n\x q\simeq (2^n\x\< 1\>)\otimes q$, we have that $i_W((2^n\x q)_K)\geqs 2^ni_1(q)$ by Theorem~\ref{i1bound}. Since $m>2^n-\frac {2^ni_1(q)} {\dim q}$, it follows that $m\x q\subset 2^n\x q$ of codimension less than $i_W((2^n\x q)_K)$. Thus, Lemma~\ref{star} implies that $m\x q$ is isotropic over $K$, whereby $\sub_q(K)\leqs m-1$.
\end{proof}



As a consequence of Theorem~\ref{inadmissiblesub}, one can see that Theorem~\ref{subvalues} does not hold for arbitrary forms.


\begin{example} Let $\pi$ be a $3$-fold Pfister form over $F$ such that $4\x \pi$ is anisotropic. Let $q$ be a $6$-dimensional neighbour of $\pi$, whereby $i_1(q)=2$ and $\sub_q(F)\geqs 4$. Hence, we can conclude that $3\notin\{\sub_q(K)\mid K/F \textrm{ field extension}\}$ by Theorem~\ref{inadmissiblesub}. Similarly, if we have that $\sub_q(F)\geqs 32$ for $q$ as above, then $3,6,7,11,\ldots ,15,22, \ldots, 31 \notin\{\sub_q(K)\mid K/F \textrm{ field extension}\}$.
\end{example}

\section{Values of the $q$-level}

In analogy with the preceding section, we study the behaviour of the $q$-level with respect to field extensions. As before, we will seek to determine the entries of the set $\{s_q(K)\mid K/F \textrm{ field extension}\}$. As with the $q$-sublevel, $s_q(F)$ always belongs to this set, with the remaining entries being less than $s_q(F)$. 


Proposition~\ref{hsub} is an analogue of \cite[Proposition 3.13~(2)]{BGBM}, where Theorem~\ref{H95} was invoked to obtain a criterion for inclusion in $\{s_q(K)\mid K/F \textrm{ field extension}\}$. We state a slightly-modified version of this result below.


\begin{prop}\label{hs} (Berhuy, Grenier-Boley, Mahmoudi) Let $q$ be an anisotropic form over $F$. If an integer $m\leqs s_q(F)$ is such that $1+(m-1)\dim q\leqs 2^n< 1+m\dim q$ for some $n\in\mathbb{N}$, then we have that $m\in\{s_q(K)\mid K/F \textrm{ field extension}\}$.
\end{prop}



In our next result, we determine the integers to which Proposition~\ref{hs} applies.


\begin{prop}\label{sfigure} Let $q$ be an anisotropic form over $F$. An integer $m\leqs s_q(F)$ is such that $1+(m-1)\dim q\leqs 2^n< 1+m\dim q$ if and only if $m=\left\lceil\frac {2^n} {\dim q}\right\rceil$ for some $n\in\mathbb{N}$. In particular, $m=\left\lceil\frac {2^n} {\dim q}\right\rceil\in\{s_q(K)\mid K/F \textrm{ field extension}\}$.
\end{prop}

\begin{proof} Let $m\leqs s_q(F)$ be such that $1+(m-1)\dim q\leqs 2^n< 1+m\dim q$ for some $n\in\mathbb{N}$. Since $1+(m-1)\dim q\leqs 2^n$, it follows that $m\leqs \frac {2^n+\dim q-1} {\dim q}$. Moreover, as $2^n< 1+m\dim q$, we have that $m>\frac {2^n-1} {\dim q}$. Thus, since $m\in\mathbb{N}$, it follows that $m\geqs \left\lfloor \frac {2^n-1} {\dim q}\right\rfloor+1=\left\lfloor \frac {2^n+\dim q-1} {\dim q}\right\rfloor$. Hence, combining the above, we have that $m=\left\lfloor \frac {2^n+\dim q-1} {\dim q}\right\rfloor=\left\lceil\frac {2^n} {\dim q}\right\rceil$.

Conversely, letting $m=\left\lceil\frac {2^n} {\dim q}\right\rceil$, we clearly have that $\dim\left(\< 1\>\perp m\x q\right)> 2^n$. Moreover, as $\left\lceil \frac {2^n} {\dim q}\right\rceil=\left\lfloor \frac {2^n-1} {\dim q}\right\rfloor+1$, we also have that $\dim\left(\< 1\>\perp \left(m-1\right)\x q\right)\leqs 2^n$. 


The last statement follows from applying Proposition~\ref{hs} to $m=\left\lceil\frac {2^n} {\dim q}\right\rceil$.
\end{proof}




\begin{cor}\label{lows} Letting $q$ be an anisotropic form over $F$, we have that $\{1,2\}\subseteq\{s_q(K)\mid K/F \textrm{ field extension}\}$.
\end{cor}

\begin{proof} For $n=0$, Proposition~\ref{sfigure} implies that $1\in\{s_q(K)\mid K/F \textrm{ field extension}\}$.

For $\dim q= 2^n+k$ for some $n,k\in\mathbb{N}$ where $0\leqs k < 2^n$, Proposition~\ref{sfigure} implies that $\left\lceil\frac {2^n} {\dim q}\right\rceil=2\in\{s_q(K)\mid K/F \textrm{ field extension}\}$.
\end{proof}



\begin{cor}\label{2powers} Let $q$ be an anisotropic form over $F$ of dimension $2^n$ for some $n\in\mathbb{N}$. Then $2^k\in\{s_q(K)\mid K/F \textrm{ field extension}\}$ for all integers $k\geqs 0$ such that $2^k\leqs s_q(F)$.
\end{cor}

\begin{proof} This follows directly from Proposition~\ref{hs} or Proposition~\ref{sfigure}.
\end{proof}

With respect to certain forms $q$ and integers $m$, Proposition~\ref{sfigure} enables us to determine whether or not $m$ is in $\{s_q(K)\mid K/F\text{ field extension}\}$.

\begin{prop}\label{modmax} Let $q$ be a form over $F$ such that $\< 1\>\perp m\x q$ has maximal splitting for some integer $m< s_q(F)$. Then $m\in\{s_q(K)\mid K/F\text{ field extension}\}$ if and only if $m=\left\lceil \frac {2^n} {\dim q}\right\rceil$ for some $n\in\mathbb{N}$.
\end{prop} 

\begin{proof} For $m=\left\lceil \frac {2^n} {\dim q}\right\rceil$ for some $n\in\mathbb{N}$, Proposition~\ref{sfigure} implies that we have $m\in\{s_q(K)\mid K/F\text{ field extension}\}$. 

Conversely, let $K/F$ be such that $s_q(K)=m$. As $\< 1\>\perp (m-1)\x q$ is anisotropic over $K$, Lemma~\ref{star} implies that $\dim(\< 1\>\perp (m-1)\x q)\leqs \dim(\< 1\>\perp m\x q)-i_W({\< 1\>\perp m\x q}_K)$. We recall that $i_W({\< 1\>\perp m\x q}_K)\geqs i_1(\< 1\>\perp m\x q)$, as the form $\< 1\>\perp m\x q$ is isotropic over $K$. Thus, we have that $\dim(\< 1\>\perp (m-1)\x q)\leqs \dim(\< 1\>\perp m\x q)-i_1(\< 1\>\perp m\x q)$. Since $\< 1\>\perp m\x q$ has maximal splitting, $\dim(\< 1\>\perp m\x q)-i_1(\< 1\>\perp m\x q)= 2^n$ for some $n\in\mathbb{N}$. Hence, it follows that $\dim(\< 1\>\perp (m-1)\x q)\leqs 2^n <  \dim(\< 1\>\perp m\x q)$. Invoking Proposition~\ref{sfigure}, we have that $m=\left\lceil \frac {2^n} {\dim q}\right\rceil$.
\end{proof}



If $q$ is an $n$-fold Pfister form over $F$ for $n\geqs 0$, then for all $m\in\mathbb{N}\setminus \{ 0\}$ we have that $\< 1\>\perp m\x q$ is a neighbour of the Pfister form $2^k\x q$, where $2^{k-1}\leqs m< 2^k$. Thus, for all $m\in\mathbb{N}\setminus \{ 0\}$, the form $\< 1\>\perp m\x q$ has maximal splitting, whereby Proposition~\ref{modmax} can be applied to all positive integers $m< s_q(F)$ in this case (see Remark~\ref{6.1rec}). Unlike the situation with respect to the $q$-sublevel however, we note that this observation does not apply to forms that are similar to $q$ (in particular, see Remark~\ref{section3.3}).

As with the $q$-sublevel, Theorem~\ref{km}~$(i)$ enables us to record another criterion for the admissibility of a prescribed integer in the set $\{s_q(K)\mid K/F \textrm{ field extension}\}$.

\begin{prop}\label{KMusages} Let $q$ be an anisotropic form over $F$ and $n\in\mathbb{N}$. If $\< 1\>\perp n\x q$ is anisotropic over $F$ and $\dim q > i_1(\< 1\>\perp n\x q)- i_1(\< 1\>\perp (n-1)\x q)$, then $s_q(F( \< 1\>\perp n\x q))=n$.
\end{prop}  

\begin{proof} Given our assumptions, Theorem~\ref{km}~$(i)$ allows us to conclude that $\< 1\>\perp (n-1)\x q$ is anisotropic over $F( \< 1\>\perp n\x q)$, whereby the results follow.
\end{proof}


For $q$ a form over $F$ of dimension $1$ or $2$ (respectively $3$) such that $s_q(F)=\infty$, \cite[Question 6.1]{BGBM} asks whether all elements of $\{s_{q}(K)\mid K/F \textrm{ field extension}\}$ are of the form $2^k$ (respectively $\frac {2^{2k}+2}{3}$ or $\frac {2^{2k+1}+1}{3}$) where $k\in\mathbb{N}$. Our next result shows that this is not the case. Indeed, for all $n\in\mathbb{N}\setminus \{0\}$, there exist $n$-dimensional forms $q$ over ordered fields $F$ that can attain any prescribed integer less than their level over $F$ as their level over a suitable extension.


\begin{theorem}\label{svalues} Let $q$ be a form over a formally real field $F$ such that $\sgn_P(q)=-\dim q$ for some $P\in X_F$. Then $\{1,\ldots ,s_q(F)\}=\{s_q(K)\mid K/F \textrm{ field extension}\}$.
\end{theorem}



\begin{proof} Clearly $\{s_q(K)\mid K/F \textrm{ field extension}\}\subseteq\{1,\ldots ,s_q(F)\}$, as $s_q(K)\leqs s_q(F)$.

Let $n\in\mathbb{N}\setminus \{ 0\}$ be such that $n< s_q(F)$, whereby $\<1\>\perp n\x q$ is anisotropic over $F$. Let $K=F(\<1\>\perp n\x q)$. Since $q$ is indefinite with respect to $P$, Theorem~\ref{ELW} implies that $P$ extends to $K$. Moreover, as $|\sgn_P(\snq) | =n\dim q-1$, we have that $i_1(\snq)=1$. Thus, Theorem~\ref{km}~$(i)$ implies that $\snlq$ is anisotropic over $K$, whereby $s_q(K)=n$.
\end{proof}


\begin{remark}\label{remsandsub} One can invoke the above proof to establish that, in general, the $q$-level does not impose an upper bound on the $q$-sublevel. Let $q$ be a form over a formally real field $F$ such that $\sgn_P(q)=-\dim q$ for some $P\in X_F$. As per the above proof, for a positive integer $n< s_q(F)$ and $K=F(\<1\>\perp n\x q)$, one has that $s_q(K)=n$. As $\sgn_P(q)= -\dim q$ and $P$ extends to $K$, it follows that $\sub_q(K)=\infty$. Furthermore, there exist field extensions $L/K$ such that $s_q(L)=n$ and $\sub_q(L)$ takes arbitrarily larger finite values. In particular, letting $m\in\mathbb{N}$ be such that $1+(n-1)\dim q\leqs m\dim q\leqs 2^{r}< (m+1)\dim q$ for some $r\in\mathbb{N}$, if one sets $L=K((m+1)\x q)$, then Theorem~\ref{H95} implies that $s_q(L)=n$ and $\sub_q(L)=m$.
\end{remark}

The following remark addresses a comment at the end of Section $3.3$ in \cite{BGBM}.

\begin{remark}\label{section3.3}
Since the sums of squares in a field $F$ are precisely the elements that are positive at all orderings of $F$ (see \cite[Theorem VIII.$1.12$]{LAM}), the existence of an ordering of $F$ at which $a\in F^{\x}$ is negative implies that the length of $a$ is infinite. Let $F$ be a field such that there exists $a\in F^{\x}$ with $a$ positive at some $P\in X_F$ and $a$ negative at some $Q\in X_F$, whereby $\ell_F(a)=\infty$. Applying Theorem~\ref{svalues} to $q\simeq \< -a\>$, one sees that the length of $a$ can equal any prescribed integer over a suitable extension (in particular, for all $n\in\mathbb{N}\setminus \{ 0\}$, we have that $\ell_K(a)= s_q(K)=n$ for $K=F(\< 1\>\perp n\x q)$).
\end{remark}

As above, by considering certain forms over certain fields, we can answer \cite[Question 6.1]{BGBM} in the negative. In \cite[Corollary 3.14]{BGBM}, it is shown that all integers $2^k$ (respectively $\frac {2^{2k}+2}{3}$ and $\frac {2^{2k+1}+1}{3}$) belong to $\{s_{q}(K)\mid K/F \textrm{ field extension}\}$ for all forms $q$ over $F$ of dimension $1$ or $2$ (respectively $3$) such that $s_q(F)=\infty$ (we note that these values can be recovered by invoking Proposition~\ref{sfigure}). For $n\in\mathbb{N}\setminus \{ 0\}$, let $A_n^{\infty}=\{\text{forms }q\text{ over }F\mid \dim q=n\text{ and }s_q(F)=\infty\}$. Thus, in the spirit of \cite[Question 6.1]{BGBM}, we will also address the following related question.


%

\begin{question}\label{modified} What are the values of $\underset{q\in A_n^{\infty}}{\bigcap} \{s_q(K)\mid K/F\text{ field extension}\}$?
\end{question}

For all $n,m\in\mathbb{N}\setminus \{ 0\}$, the following example shows the existence of forms $q\in A_n^{\infty}$ such that $\< 1\>\perp m\x q$ has maximal splitting.

\begin{example}\label{nones} As $s_q(F)=\infty$ for all $q\in A_n^{\infty}$, we have that $\sub_q(F)=\infty$ for all $q\in A_n^{\infty}$. Hence, we can conclude that $F$ is formally real. Thus, $n\x\< 1\>\in A_n^{\infty}$ for all $n\in\mathbb{N}\setminus \{ 0\}$. Moreover, for all $m\in\mathbb{N}\setminus \{ 0\}$, the form $\< 1\>\perp m\x(n\x\< 1\>)$ is a Pfister neighbour of $2^r\x\< 1\>$, for $2^{r-1}\leqs mn<2^r$, whereby it has maximal splitting.
\end{example}

Thus, invoking our earlier results, we have the following answer to Question~\ref{modified}:

\begin{prop}\label{mset} A positive integer $m$ is in $\underset{q\in A_n^{\infty}}{\bigcap} \{s_q(K)\mid K/F\text{ field extension}\}$ if and only if $m=\left\lceil \frac {2^r} {\dim q}\right\rceil$ for some $r\in\mathbb{N}$.
\end{prop}


\begin{proof} Let $m\in\underset{q\in A_n^{\infty}}{\bigcap} \{s_q(K)\mid K/F\text{ field extension}\}$. Since $m=s_q(K)$ for some $K/F$ and $q\in A_n^{\infty}$ such that $\< 1\>\perp m\x q$ has maximal splitting, Proposition~\ref{modmax} implies that $m=\left\lceil \frac {2^r} {\dim q}\right\rceil$ for some $r\in\mathbb{N}$.

Conversely, invoking Proposition~\ref{sfigure}, we see that $m=\left\lceil \frac {2^r} {\dim q}\right\rceil$ belongs to the set $\underset{q\in A_n^{\infty}}{\bigcap} \{s_q(K)\mid K/F\text{ field extension}\}$.
\end{proof}

Whereas $\sub_q(F)\leqs s_q(F)+s_{-q}(F)-1$ for $q$ an anisotropic form over $F$, Remark~\ref{remsandsub} demonstrates that finiteness of $s_q(F)$ does not imply that of $\sub_q(F)$. Thus, in general, we cannot hope to invoke Theorem~\ref{i1bound} to show that certain integers do not belong to $\{s_q(K)\mid K/F \textrm{ field extension}\}$ (indeed, as per Theorem~\ref{svalues}, there exist anisotropic forms $q$ that can take any prescribed integer as their level over a suitable extension). For those forms $q$ such that finiteness of their level implies finiteness of their sublevel, one can attempt to establish an analogue of Theorem~\ref{inadmissiblesub}. In this regard, we formulate the following result with respect to forms that represent one.





\begin{prop}\label{inadmissibles} Let $q$ be an anisotropic form over F such that $1\in D_F(q)$ and $\sub_q(F)\geqs 2^n$ for some $n\in\mathbb{N}$. If $\dim q\leqs 2^{n-1}i_1(q)$, then for all integers $m\in\left( 2^n+1-\left(\frac {2^ni_1(q)+1} {\dim q}\right), 2^n\right)$ we have that $m\notin \{s_{q}(K)\mid K/F \textrm{ field extension}\}$.
\end{prop}


\begin{proof} The condition that $\dim q\leqs 2^{n-1}i_1(q)$ guarantees the existence of at least one integer $m\in\left( 2^n+1-\left(\frac {2^ni_1(q)+1} {\dim q}\right), 2^n\right)$.

Let $m$ be an integer such that $m\in\left( 2^n+1-\left(\frac {2^ni_1(q)+1} {\dim q}\right), 2^n\right)$. Since $m\leqs 2^n-1$ and $1\in D_F(q)$, it follows that $\< 1\>\perp m\x q\subset 2^n\x q$. As $2^n\leqs \sub_q(F)$, we have that $2^n\x q$ is anisotropic over $F$.  Let $K/F$ be such that $\< 1\>\perp m\x q$ is isotropic over $K$, whereby $2^n\x q$ is isotropic over $K$. Since $2^n\x q\simeq (2^n\x\< 1\>)\otimes q$, we have that $i_W((2^n\x q)_K)\geqs 2^ni_1(q)$ by Theorem~\ref{i1bound}. Since $m> 2^n+1-\left(\frac {2^ni_1(q)+1} {\dim q}\right)$, we have that $\dim(\< 1\>\perp (m-1)\x q)> 2^n\dim q-2^n i_1(q)$. Hence, Lemma~\ref{star} implies that $\< 1\>\perp (m-1)\x q$ is isotropic over $K$, whereby $s_q(K)\leqs m$.
\end{proof}

As was the case with Theorem~\ref{inadmissiblesub}, through the consideration of such forms with maximal splitting and such forms of dimension a power of two, one can show that the interval in the above result cannot be extended in either direction.

\section{The weak-isotropy index of round and Pfister forms}

In \cite[Proposition $4.1$]{BGBM}, it was shown that $\sub_q(F)=s_q(F)$ for $q$ an anisotropic group form over $F$. Thus, in our considerations of Pfister forms $q$, we will state all results in terms of $s_q(F)$, bearing in mind that the same statements hold for $\sub_q(F)$.

In \cite[Proposition $4.3$]{BGBM}, it was shown that the $s_q(F)$ is either a power of two or is infinite in the case where $q$ is a round form over $F$. Furthermore, in the case where $q$ is a Pfister form over $F$, the following result \cite[Theorem $4.4$]{BGBM} was established.
 
\begin{theorem}\label{pfisters} (Berhuy, Grenier-Boley, Mahmoudi) If $q$ is an anisotropic Pfister form over $F$, then $\{s_q(K)\mid K/F \textrm{ field extension}\}=\{1,\ldots ,2^{i},\ldots ,s_q(F)\}$.
\end{theorem}

\begin{remark}\label{6.1rec} As before, if $q$ is an anisotropic $n$-fold Pfister form over $F$ for $n\geqs 0$, then for all $m\in\mathbb{N}$ we have that $\< 1\>\perp m\x q$ is a Pfister neighbour, and thus has maximal splitting. Hence, the above result can be recovered by invoking Proposition~\ref{submodmax}. 
\end{remark}

In \cite[Remark $4.5$]{BGBM}, the question of whether Theorem~\ref{pfisters} holds for round forms was considered. We next show that this is not the case.



\begin{example} Consider the round form $q\simeq\< 1,1,1\>$ over $F=\mathbb{R}$. Theorem~\ref{H95} implies that $\< 1\>\perp 2\x q$ is anisotropic over $F(\< 1\>\perp 3\x q)$, whereby we can conclude that $3\in\{s_q(K)\mid K/F \textrm{ field extension}\}$. Moreover, letting $K/F$ be such that $\< 1\>\perp 4\x q$ is isotropic over $K$, it follows that the Pfister form $16\x \< 1\>$ is hyperbolic over $K$. Hence, Lemma~\ref{star} implies that $\< 1\>\perp 3\x q$ is isotropic over $K$, and thus that $4\notin\{s_q(K)\mid K/F \textrm{ field extension}\}$. Since group forms do not necessarily remain group forms after passing to a field extension, one must differentiate between $\sub_q(K)$ and $s_q(K)$ for a general extension $K/F$. By employing the same arguments as above however, one can establish that $4\notin\{\sub_q(K)\mid K/F \textrm{ field extension}\}$ and $5\in\{\sub_q(K)\mid K/F \textrm{ field extension}\}$.
\end{example}


For $\tau$ a neighbour of a Pfister form $q$ over $F$, it was shown in \cite[Corollary $4.6$]{BGBM} that $s_q(F)\leqs s_{\tau}(F)\leqs 2s_q(F)$. As a Pfister neighbour is not a group form in general, its sublevel need not equal its level, and thus the value of its sublevel merits a separate treatment. In this regard, we can establish the following relations between the sublevel of a Pfister neighbour and that of its associated Pfister form.




\begin{prop}\label{subPN} Let $q$ be an anisotropic Pfister form over $F$ and let $\tau$ be a neighbour of $q$. Then $\sub_{q}(F)\leqs \sub_{\tau}(F)\leqs \left\lfloor\frac {\sub_{q}(F)\dim q} {\dim\tau}\right\rfloor$.
\end{prop}

\begin{proof} Proposition~\ref{basicsub}~$(ii)$ implies that $\sub_{q}(F)\leqs \sub_{\tau}(F)$. 

In order to prove the remaining inequality, we may assume that $\sub_{q}(F)<\infty$, whereby \cite[Proposition $4.3$]{BGBM} implies that $\sub_{q}(F)=2^n$ for some $n\in\mathbb{N}$. Thus, the Pfister form $2^{n+1}\x q$ is hyperbolic over $F$. Since $\left\lfloor\frac {\sub_{q}(F)\dim q} {\dim\tau}\right\rfloor\leqs 2^{n+1}-1$, we have that $\left(\left\lfloor\frac {\sub_{q}(F)\dim q} {\dim\tau}\right\rfloor +1\right)\x\tau\subset 2^{n+1}\x q$. Since $\left\lfloor\frac {\sub_{q}(F)\dim q} {\dim\tau}\right\rfloor =\left\lceil\frac {\sub_{q}(F)\dim q +1} {\dim\tau}\right\rceil-1$, we moreover have that $\dim\left(\left(\left\lfloor\frac {\sub_{q}(F)\dim q} {\dim\tau}\right\rfloor +1\right)\x\tau\right) \geqs\sub_{q}(F)\dim q +1=2^n\dim q+1$, whereby Lemma~\ref{star} implies that $\vartheta$ is isotropic, establishing the result.
\end{proof}

Clearly, the above bounds are attained in the case where $\tau$ is similar to $q$.

As a corollary of the above result, we can make the following statement regarding the level of a Pfister neighbour. 

\begin{cor}\label{sPN} Let $q$ be an anisotropic Pfister form over $F$ and let $\tau$ be a neighbour of $q$. Then $s_q(F)\leqs s_{\tau}(F)\leqs \left\lceil\frac {s_{q}(F)\dim q+1} {\dim\tau}\right\rceil$.
\end{cor}

\begin{proof} Given \cite[Corollary $4.6$]{BGBM}, we need only establish the upper bound on $s_{\tau}(F)$. Since $s_{\tau}(F)\leqs \sub_{\tau}(F)+1$, Proposition~\ref{subPN} implies that $s_q(F)\leqs s_{\tau}(F)\leqs \left\lfloor\frac {\sub_{q}(F)\dim q} {\dim\tau}\right\rfloor+1$. Since $\left\lfloor\frac {\sub_{q}(F)\dim q} {\dim\tau}\right\rfloor = \left\lceil\frac {s_{q}(F)\dim q+1} {\dim\tau}\right\rceil -1$, the result follows.
\end{proof}

Clearly, the lower bound in the above result is attained in the case where $\tau$ is similar to $q$. As per \cite[Example $4.7$]{BGBM}, for $p\neq 2$ a prime number and $F=\mathbb{Q}_p$, the upper bound can be attained by letting $q$ be the unique anisotropic $4$-dimensional form over $F$ with pure subform $\tau$, whereby $s_{\tau}(F)=2$ and $s_q(F)=1$.

In \cite[Lemma $4.8$]{BGBM}, for $q$ a Pfister form over $F$, the following lower bound on the $q$-level over a quadratic extension of $F$ was established. Addressing \cite[Remark $4.23$]{BGBM}, we show that this result holds for round forms. Additionally, we establish a related upper bound.



%

\begin{prop}\label{4.8} Let $q$ be an anisotropic round form over $F$ and let $K=F(\sqrt{d})$ be a quadratic field extension of $F$. Then we have that $s_q(K)\leqs\ell_{q}(-d)\leqs 2s_{q}(K)$.
\end{prop}

\begin{proof} As $q$ is a round form over $F$, we have that $1\in D_F(q)$. Hence, we have that $-d\in D_F(\ell_{q}(-d)\x q)$, and thus that $-1\in D_K(\ell_{q}(-d)\x q)$. Thus, $s_q(K)\leqs\ell_{q}(-d)$.


To establish the lower bound on $s_q(K)$, we will consider two cases: that where $s_q(K)=s_q(F)$, and that where $s_q(K)< s_q(F)$.

Since $1\in D_F(q)$, we note that $(s_q(K)+1)\x q$ is isotropic over $F$ (and hence universal) in the case where $s_q(K)=s_q(F)$. Thus, $\ell_q(-d)\leqs s_q(K) +1\leqs 2s_q(K)$. 

If $s_q(K)< s_q(F)$, then $\< 1\>\perp s_q(K)\x q$ is anisotropic over $F$ and isotropic over $K$. Thus, there exists $a\in F^{\x}$ such that $a\< 1,-d\>\subset \< 1\>\perp s_q(K)\x q$ by \cite[Theorem VII.$3.1$]{LAM}. Letting $s_q(K)=2^n+k$ for $0\leqs k<2^n$, we have that $a\< 1,-d\>\subset 2^{n+1}\x q$ since $1\in D_F(q)$. By Witt's Round Form Theorem, $2^{n+1}\x q$ is a round form. Hence, since $a\in D_F(2^{n+1}\x q)$, we can conclude that $\< 1,-d\>\subset a(2^{n+1}\x q)\simeq 2^{n+1}\x q$. Thus, $\ell_q(-d)\leqs 2^{n+1}\leqs 2s_q(K)$. 
\end{proof}





Thus, we have the following analogue of \cite[Proposition $4.10$]{BGBM} for round forms.




\begin{prop}\label{4.10} Let $q$ be an anisotropic round form over $F$ and let $d\in F$ be an element such that $\ell_{q}(-d)=n$. If $K=F(\sqrt{d})$, we have that $2^{r-1}\leqs s_q(K)< 2^{r+1}$, where $r$ is determined by $2^r\leqs n< 2^{r+1}$.
\end{prop}

\begin{proof} This follows as an immediate corollary of Proposition~\ref{4.8}.
\end{proof}

In analogy with the $q$-length of $a\in F^{\x}$, for $q$ and $\varphi$ forms over $F$, we define the \emph{$q$-length of $\varphi$} to be $\ell_q(\varphi)=\min\{n\in\mathbb{N}\mid \varphi\subset n\x q\}$ if such numbers $n$ exist, and set it to be infinite otherwise. Returning to the case where $q$ is a Pfister form over $F$, we can strengthen the above results and establish these generalised statements with respect to arbitrary function fields of quadratic forms.


\begin{prop}\label{4.8Pfister} Let $q$ be an anisotropic Pfister form over $F$ and $\varphi$ an anisotropic form over $F$ such that $\dim\varphi\leqs (s_q(F))\dim q$. Then $s_{q}(F(\varphi))\leqs\ell_{q}(a\varphi)\leqs 2s_{q}(F(\varphi))$ for every $a\in D_F(\varphi)$.
\end{prop}

\begin{proof}  Since $a\varphi\subset \ell_{q}(a\varphi)\x q$, the form $ \ell_{q}(a\varphi)\x q$ is isotropic over $F(\varphi)$, whereby $s_{q}(F(\varphi))\leqs\ell_{q}(a\varphi)$.

To establish the lower bound on $s_q(F(\varphi))$, we will separately consider the case where $s_{q}(F(\varphi))=s_q(F)$ and that where $s_q(F(\varphi))< s_q(F)$.

If $s_q(F(\varphi))= s_q(F)$, then the Pfister form $2s_{q}(F(\varphi))\x q$ is hyperbolic over $F$. Thus, for $a\in D_F(\varphi)$, we have that $a\varphi\subset a\varphi\perp -a\varphi\subset 2s_{q}(F(\varphi))\x q$ since $\dim\varphi\leqs (s_q(F))\dim q$, establishing the result in this case.

If $s_q(F(\varphi))< s_q(F)$ on the other hand, then we have that $2s_q(F(\varphi))\leqs s_q(F)$ by \cite[Proposition $4.3$]{BGBM}. As $2s_q(F(\varphi))\x q$ is a Pfister form, it follows that it is anisotropic over $F$, as otherwise Lemma~\ref{star} would imply that $\< 1\>\perp s_q(F(\varphi))\x q$ is isotropic over $F$, a contradiction in this case. Since $\< 1\>\perp s_q(F(\varphi))\x q$ is isotropic over $F(\varphi)$, it follows that $2s_q(F(\varphi))\x q$ becomes hyperbolic over $F(\varphi)$. Thus, invoking \cite[Theorem X.$4.5$]{LAM}, we have that $a\varphi\subset 2s_q(F(\varphi))\x q$ for every $a\in D_F(\varphi)$, establishing the result.
\end{proof}



The following example shows that the above bounds can be attained.

\begin{example} If $a\varphi\subset q$, then clearly $s_{q}(F(\varphi))=\ell_{q}(a\varphi)=1$. Next, let $F$ be a field of $q$-level at least two. If $a\varphi\simeq 2\x q$, then $\ell_{q}(a\varphi)=2$. Moreover, as the Pfister form $2\x q$ is hyperbolic over $F(\varphi)$ in this case, we have that $\< 1\>\perp q$ is isotropic over $F(\varphi)$ by Lemma~\ref{star}. Thus, we have that $\ell_{q}(a\varphi)= 2s_{q}(F(\varphi))$ for $a\varphi\simeq 2\x q$.
\end{example}



\begin{prop}\label{4.10Pfister} Let $q$ be an anisotropic Pfister form over $F$ and $\varphi$ an anisotropic form over $F$ such that $\dim\varphi\leqs (s_q(F))\dim q$. If $\ell_{q}(a\varphi)=n$ for some $a\in D_F(\varphi)$ where $2^{r}< n\leqs 2^{r+1}$, then $s_q(F(\varphi))=2^{r}$.
\end{prop}

\begin{proof} We will first prove that $s_q(F(\varphi))\leqs 2^{r}$. If $s_q(F)\leqs 2^{r}$, then this is clear. Hence, we may assume that $s_q(F)\geqs 2^{r+1}$. In this case, the Pfister form $2^{r+1}\x q$ is anisotropic over $F$, as otherwise Lemma~\ref{star} would imply that $\< 1\>\perp 2^{r}\x q$ is isotropic over $F$, a contradiction. Since $a\varphi\subset n\x q$ for some $a\in D_F(\varphi)$, we have that $2^{r+1}\x q$ is hyperbolic over $F(\varphi)$, and thus that $\< 1\>\perp 2^{r}\x q$ is isotropic over $F(\varphi)$ by Lemma~\ref{star}. Thus, we have that $s_q(F(\varphi))\leqs 2^{r}$.

Proposition~\ref{4.8Pfister} implies that $n\leqs 2s_{q}(F(\varphi))$. As $s_{q}(F(\varphi))$ is necessarily a $2$-power by \cite[Proposition $4.3$]{BGBM}, we can conclude that $2^{r+1}\leqs 2s_{q}(F(\varphi))$.

Combining the above, it follows that $s_q(F(\varphi))=2^{r}$.
\end{proof}

{\large \textbf{Acknowledgements.}} I gratefully acknowledge the support I received through an International Mobility Fellowship from the Irish Research Council, co-funded by Marie Curie Actions under FP7.



%

{


}

\end{document}